\documentclass[pdflatex,sn-mathphys-num]{sn-jnl}

\usepackage{graphicx}%
\usepackage{multirow}%
\usepackage{amsmath,amssymb,amsfonts}%
\usepackage{amsthm}%
\usepackage{mathrsfs}%
\usepackage[title]{appendix}%
\usepackage{xcolor}%
\usepackage{textcomp}%
\usepackage{manyfoot}%
\usepackage{booktabs}%
\usepackage{algorithm}%
\usepackage{algorithmicx}%
\usepackage{algpseudocode}%
\usepackage{listings}%
\usepackage{comment}
\usepackage{subcaption}

\theoremstyle{thmstyleone}%
\newtheorem{theorem}{Theorem}
\newtheorem{proposition}[theorem]{Proposition}%

\theoremstyle{thmstyletwo}%

\theoremstyle{thmstylethree}%

\begin{document}

\title[Article Title]{Bounds on Intrinsic Bayes Factors and Least Favorable Intrinsic Priors for General Statistical Hypothesis Testing}


\author*[1,2]{\fnm{Fernando E.} \sur{Betancourt}}\email{fernando.betancourt@upr.edu}

\author[3]{\fnm{Richard} \sur{Clare}}\email{nixordesigns@gmail.com}
\equalcont{These authors contributed equally to this work.}

\author[1]{\fnm{Luis R.} \sur{Pericchi}}\email{luis.pericchi@upr.edu}
\equalcont{These authors contributed equally to this work.}

\affil[1]{\orgdiv{Department of Mathematics}, \orgname{University of Puerto Rico }, \orgaddress{\street{Rio Piedras}, \city{San Juan}, \postcode{00925}, \state{PR}, \country{USA}}}

\affil[2]{\orgdiv{Department of Mathematics \& Physics}, \orgname{University of Puerto Rico }, \orgaddress{\street{} \city{Cayey}, \postcode{00736}, \state{PR}, \country{USA}}}

\affil[3]{\orgdiv{Infinite Trading LLC}, \orgname{} \orgaddress{\street{} \city{San Juan}, \postcode{00925}, \state{PR}, \country{USA}}}

\abstract{Hypothesis Testing is the most contentious procedure in statistical Methodology. P values rejects Null Hypotheses far too easily, specially for large samples. On the other hand, Bayes Factors depends on assumptions, for example regarding Intrinsic Bayes Factors, which average? Arithmetic, Geometric, Median? Our bound is the infimum over all the averages. We develop a lower bound on Intrinsic Bayes Factors that adjust authomatically with the sample size. Furthermore, we introduce the new idea of {\it{\textbf{Least Favorable Intrinsic Prior}}}, which corresponds to the least favourable possible training samples. The bound sets a bridge between Intrinsic Bayes Factors and Adrian Smith and David Spiegelhalter methodology.}




\maketitle

\section{Essence of the method}

 We propose a novel method that generates a lower bound for the
Bayes factor and corresponding priors, employing an approach that is closely related
to the concept of Intrinsic Bayes Factors and the technique employed by
\cite{spiegelhalter1982bayes}.
In general, Improper Priors can not be used for Bayes Factors. One powerful way out of it, is the use of real or imaginary "training samples".  
Let $y(\ell)$ denote a minimal training sample (whose posteriors are proper priors for all models), and let $\pi(\theta_i \mid y(\ell))$
represent the posterior distribution of parameter $\theta_i$, given the minimal training
sample for two models, $M_i$ and $M_j$. Our objective is to compare these two models,
$M_i$ and $M_j$. To achieve this, we define the set $D$ as the collection of all possible
minimal training samples of size $k$, where $k = \max\{\dim(\theta_i),\dim(\theta_j)\}.$
In this paper, we will explore various approaches to construct lower bounds for the
Intrinsic Bayes factor, ultimately contributing to the field of model comparison and
selection. There is a bridge between IBF and the Least Favorable Intrinsic Bayes Factor
approach. The starting point is the following Basic Lemma 1 from \cite{berger1996intrinsic}.\\
\[
B_{10}(y(-l)|y(l))=B^N_{10}({\textbf{y}})\times B^N_{01}(y(l)),
\]
where $y(-l)$ is the complement of the training sample. The dependence on the specific training sample is removed taking averages. Special status has the Arithmetic IBF ((asmptotically equivalent to to Intrinsic Priors): 
\begin{equation}
B_{10}^{AI}({\textbf{y}})
=
B_{10}^{N}({\textbf{y}})\times \frac{1}{L}\sum_{\ell=1}^{L} B_{01}^{N}\!\bigl(y(\ell)\bigr)
\tag{3}
\end{equation}
where $L=\binom{n}{m}$, where $m$ is the minimal training sample size. The AIBF is asymmetric;
the more complex model has to be placed above to guarantee the convergence. Before we continue we want to make a distinction in the notation that follows:\\
\textbf{Notation}: Let $D$ be the set of all possible empirical training samples $y(\ell)$ and let $D^*$ be the set of all possible theoretical training samples $y^*(\ell)$

We can first define the empirical Upper Bound (\(UB_E\)) and Lower Bound
(\(LB_E\)) as
\begin{equation}
UB_E
:=
B_{10}^{N}(\mathbf{y})
\max_{\ell=1,\ldots,L}
B_{01}^{N}\!\bigl(y(\ell)\bigr)
\ge
B_{10}^{N}(\mathbf{y})
\frac{1}{L}
\sum_{\ell=1}^{L}
B_{01}^{N}\!\bigl(y(\ell)\bigr)
=
B_{10}^{AI}(\mathbf{y}).
\tag{4}
\end{equation}

Similarly,
\begin{equation}
\begin{aligned}
LB_E
&:=
B_{01}^{N}(\mathbf{y})
\min_{\ell=1,\ldots,L}
B_{10}^{N}\!\bigl(y(\ell)\bigr) \\
&=
B_{01}^{N}(\mathbf{y})
\frac{1}{
\displaystyle
\max_{\ell=1,\ldots,L}
B_{01}^{N}\!\bigl(y(\ell)\bigr)
}
\\
&\le
B_{01}^{N}(\mathbf{y})
\frac{1}{
\displaystyle
\frac{1}{L}
\sum_{\ell=1}^{L}
B_{01}^{N}\!\bigl(y(\ell)\bigr)
}
=
B_{01}^{AI}(\mathbf{y}).
\end{aligned}
\tag{5}
\end{equation}

By definition, \cite{berger1996intrinsic}
\(B_{01}^{AI}=1/B_{10}^{AI}\). Moreover,
\(UB_E\ge B_{10}^{AI}\), since the maximum of the empirical correction factors
is necessarily greater than or equal to their arithmetic average. By taking
reciprocals, we similarly obtain
\[
B_{01}^{AI}\ge LB_E.
\]

We can now generalize these empirical bounds by enlarging the optimization
domain from the observed minimal training samples \(y(\ell)\) to the full
theoretical training-sample space \(D^*\). Let \(y^*(\ell)\) denote a
theoretical or imaginary training sample. Define
\begin{equation}
UB_T
:=
B_{10}^{N}(\mathbf{y})
\sup_{y^*(\ell)\in D^*}
B_{01}^{N}\!\bigl(y^*(\ell)\bigr),
\qquad
LB_T
:=
B_{01}^{N}(\mathbf{y})
\inf_{y^*(\ell)\in D^*}
B_{10}^{N}\!\bigl(y^*(\ell)\bigr).
\tag{6}
\end{equation}

Since the empirical training samples form a subset of the possible theoretical
training samples,
\[
\max_{\ell=1,\ldots,L}
B_{01}^{N}\!\bigl(y(\ell)\bigr)
\le
\sup_{y^*(\ell)\in D^*}
B_{01}^{N}\!\bigl(y^*(\ell)\bigr).
\]
Consequently,
\[
UB_E\le UB_T.
\]
By taking reciprocals, the ordering is reversed for the lower bounds
\[
LB_T\le LB_E.
\]

Therefore, the complete ordering is
\[
B_{10}^{AI}
\le
UB_E
\le
UB_T,
\]
and, equivalently,
\[
LB_T
\le
LB_E
\le
B_{01}^{AI}.
\]

Thus, the empirical bound is obtained directly from the observed minimal
training samples, whereas the theoretical bound is obtained by extending the
optimization to all possible minimal training samples in $D^*$. The empirical
lower bound is therefore tighter than the theoretical lower bound, while the
theoretical lower bound represents the most conservative bound over the full
training sample space.

In addition, the results
obtained with the Arithmetic Intrinsic Bayes Factor (AIBF) can be easily generalized to other
related Bayes Factors such as the Geometric (GIBF), Expected (EIBF), Harmonic (HIBF), and
Median (MIBF) Intrinsic Bayes Factors.

\section{Motivating Example, for IBF Lower Bounds}
\subsection{Normal Precision $h$ example: $H_0: h = h_0$ vs $H_1: h \neq h_0$, $\mu$ unknown}

Consider i.i.d. samples $\mathbf{y} = \{y_1, \dots, y_n\}$ where $y_i \sim N(\mu, h^{-1})$. To test $H_0: h=h_0$ against $H_1: h\neq h_0$ with $\mu$ unknown.\\
Applying Jeffreys' rule, the prior is $\pi(\mu,h) = C/h^r$, where $r=1/2$ for the dependent case and $r=1$ for the independent case ($\pi_I^J(\mu,h) \propto h^{-1}$). Under $H_0$, we assume $\pi^J_0(\mu) = C_0$. Adopting the independent prior ($r=1$) and integrating out the parameters, the Bayes Factor $B_{10}$ is:
\begin{equation}\label{B10}
B^N_{10}(\mathbf{y}) = \frac{C_1}{C_0} \left(\frac{2}{h_0 S_n^2}\right)^{\frac{n-1}{2}} \Gamma\left(\frac{n-1}{2}\right) \exp\left(\frac{h_0 S^2_n}{2}\right) , \quad S_n^2=\sum_{i=1}^n (y_i-\bar{y})^2
\end{equation}
In this context, we proceed with the independent prior $(r=1)$.

It is apparent from this expression that the Bayes Factor is undetermined, since it depends on the undefined ratio $C_1/C_0$, the undefined constants that come from the improper priors. It can be argued, from several points of view, that the constants with respect to location $\mu$ cancel out, leaving only the indeterminacy related to the hypothesis parameter $h$. Some of these points of view are, among others:
\begin{enumerate}
    \item Mean and Precision parameters are orthogonal in the Fisher Information Matrix, and thus "cancel-out" in the ratio of marginal likelihoods
    \item It turns out that two location models are predictively matched; see, for example, \cite{pericchi2005model} and when the scale is integrated into the denominator then models under both hypotheses become location models, and the corrections cancel out. \textbf{Important note}: this predictive matching property will be lost for larger than minimal training samples.
\end{enumerate}
 It is evident that the Bayes factor remains defined only up to the constant that pertains to the precision parameter under test. This is why Jeffrey's suggested conventional 
proper priors for the extra-parameters under the larger hypothesis (but improper for common parameters like $\mu$ here). There have been suggested techniques around this problem in \cite{smith1980bayes} and \cite{spiegelhalter1982bayes}. These techniques, although approximate, are useful in devising sensible scaling for Bayes Factors. In fact, this approach has a very direct relationship with more recent approaches that have been studied in detail, particularly the Intrinsic Bayes Factor, the Intrinsic Priors, and EP-priors.\\

In the example of the section, the minimal training sample consists of two observations $(y^*(\ell_1),y^*(\ell_2))$ and using Fundamental Lemma the Bayes Factor turns out to be
\begin{equation}\label{B10}
B^N_{10}(\mathbf{y})\cdot B^N_{0,1}(\mathbf y^*(\ell))=B^N_{10}\cdot (\frac{h_0 S^2_2}{2})^{1/2}\frac{\exp(-h_0 S^2_2/2)}{\Gamma(1/2)}\frac{C_0}{C_1}
\end{equation}
  Note that $S^2_2=\frac{d^*(\ell)^2}{2}$  where $d^*(\ell)=(y^*(\ell_1)-y^*(\ell_2)$.) It is apparent from (\ref{B10}) that the undefined constants cancel out. 
  The crucial step is: what to do about the (theoretical or imaginary) training samples summary statistics $d^*(\ell)$? Our practical and simplifying approach is to take: 
    \begin{equation}\label{Doptimal}
    \sup_{y^*(\ell) \in D^*} B^N_{01}(d^*(\ell)),
    \end{equation}
    which is attained at $\hat{d}^*=\sqrt{2/h_0}$. We can find this by looking at the minimum of $B^N_{10}(y^*(\ell)),$ 
    $$ B_{10}^N(y^*(\ell)) \propto \frac{1}{\exp\{-h_0 d^*(\ell)^2/4\}|d^*(\ell)|} = \frac{\exp\{h_0d^*(\ell)^2/4\}}{|d^*(\ell)|}$$

Taking logarithms on both sides and differentiating, we obtain the result. The lower bound is 
\[
LB_T(y)
=
\sqrt{2\pi e}\,
\left(\frac{h_0 S_n^2}{2}\right)^{(n-1)/2}
\frac{\exp\!\left(-\frac{h_0 S_n^2}{2}\right)}
{\Gamma\!\left(\frac{n-1}{2}\right)}.
\]
With this assignment, we argue this is a robust bound on the Bayes Factor in favor of the Null hypotheses. 
    To see the relationship with Intrinsic Bayes Factors the $\sup$ above is replaced by the arithmetic mean or by the theoretical expectation, on $y(\ell)$, under $H_1 \supset H_0$. 
\subsubsection{Comparing with Robust Bound}
We now compare with the Robust Bound \cite{sellke2001calibration}: 
    \[
    B_{01} \ge -ep\log(p)=LB_{01}(p).
    \]
    In order to calculate a suitable p-value, let's form the Likelihood Ratio:
\[
LR_{01}=(S^2_n h_0)^{n/2} \exp(-\frac{n}{2} (h_0 S^2_n - 1))
\]
Thus the Likelihood Ratio test Rejects iff $h_0 S^2_n>c_1$ or $h_0 S^2_n<c_2$, and $S^2_n h_0 \sim \chi^2(n-1)$ under $H_0$. Therefore a Likelihood Ratio test is equivalent to: $h_0 S^2_n>\chi^2_{1-\alpha/2}(n-1)$ or $h_0 S^2_n<\chi^2_{\alpha/2}(n-1)$. In Figure 1, we compare for fixed $p=0.05$ and sample size from 10 to 30, the (fixed) Robust Lower Bound (on $B_{01}$): $ -\exp(1)\cdot p \cdot \log(p)$, the $LB$ and Intrinsic Bayes Factor .  It can be argued that IBF and LB have the same asymptotic behavior, there is a small difference in that IBF is more precise. Still LB improves with $n$ and robust bound does not.  

\begin{figure}[h]
 \centering   \includegraphics[scale=0.22]{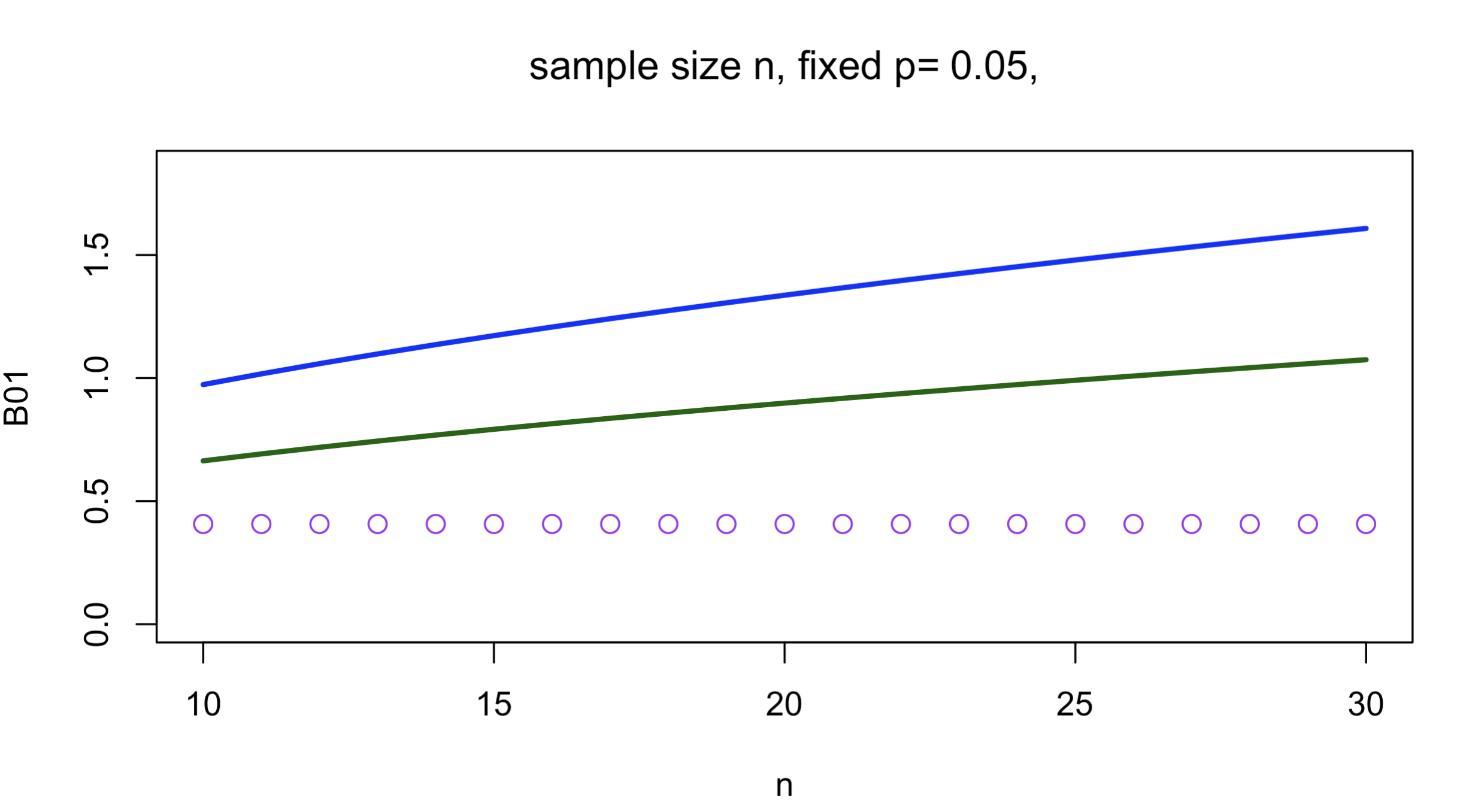}
    \caption{Three Different Lower Bounds on the Bayes Factors for fixed p and changing n:\\ $-ep\log(p)$(purple),  $LB$ (green) and IBF (blue)}
\end{figure}


\subsection{Normal mean hypothesis test example}

Assume that $y_i \sim N(\mu,\sigma_0^2)$ and consider the following hypothesis test
\[
H_0:\ \mu=\mu_0
\quad \text{vs} \quad
H_1:\ \mu\neq \mu_0,\ \sigma_0\ \text{known}.
\]
Hence, the non-informative prior is $\pi^{N}(\mu)=1$ and a minimal training sample is only
one observation $y^*(\ell)$, for any $\ell=1,\dots,n$.

\[
B_{01}^{N}\bigl(y^*(\ell)\bigr)
=
\frac{\dfrac{1}{\sqrt{2\pi}\sigma_0}
\exp\!\left\{-\frac{1}{2\sigma_0^{2}}\bigl(y^*(\ell)-\mu_0\bigr)^{2}\right\}}
{\displaystyle \int_{-\infty}^{\infty}\frac{1}{\sqrt{2\pi}\sigma_0}
\exp\!\left\{-\frac{1}{2\sigma_0^{2}}\bigl(y^*(\ell)-\mu\bigr)^{2}\right\}\,d\mu }.
\]

From above we can obtain that when $y^*(\ell)=\mu_0$ and $\sup B_{01}^{N}\bigl(y^*(\ell)\bigr)=\dfrac{1}{\sqrt{2\pi}\sigma_0}.$
Using this result we can compute the Bayes factor upper bound is given by
\[
UB_T
=
\frac{
\displaystyle \int_{-\infty}^{\infty}
\left(\frac{1}{\sqrt{2\pi}\sigma_0}\right)^{n}
\exp\!\left\{-\frac{1}{2\sigma_0^{2}}\left[
\sum_{\ell=1}^{n}\bigl(y(\ell)-\bar y\bigr)^{2}+n(\bar y-\mu)^{2}\right]\right\}\,d\mu
}{
\left(\frac{1}{\sqrt{2\pi}\sigma_0}\right)^{n}
\exp\!\left\{-\frac{1}{2\sigma_0^{2}}\sum_{\ell=1}^{n}\bigl(y^*(\ell)-\mu_0\bigr)^{2}\right\}
}
\times \sup B_{01}^{N}\bigl(y^*(\ell)\bigr).
\]

Simplifying the equation above we obtain
\[
UB_T
=
\frac{\sqrt{2\pi}\,\dfrac{\sigma_0}{\sqrt{n}}}{
\exp\!\left\{-\frac{1}{2\sigma_0^{2}}\,n(\bar y-\mu_0)^{2}\right\}
}
\,\sup B_{01}^{N}\bigl(y^*(\ell)\bigr)
=
\frac{\sqrt{2\pi}\,\dfrac{\sigma_0}{\sqrt{n}}}{
\exp\!\left\{-\frac{1}{2\sigma_0^{2}}\,n(\bar y-\mu_0)^{2}\right\}
}
\cdot \frac{1}{\sqrt{2\pi}\sigma_0}.
\]
Hence, by taking reciprocals  the IBF lower bound is obtained. We can write this lower bound in terms of the $p$-value. First notice that in terms of the $z$ statistic 
\[
LB_T=\sqrt{n}\exp{\{-z^2/2\}}=\sqrt{n}\exp{\{-qchisq(1-p,1)/2\}}
\]
where $qchisq()$ is the quantile function of the $\chi^2$ distribution. This formula is comparable to the bound in \cite{sellke2001calibration} , but this bound improves with the amount of information.  
\subsection{Normal hypothesis test $H_0:\mu=\mu_0$ vs $H_1:\mu\neq\mu_0$, $\sigma$ unknown}

Consider i.i.d. observations $Y_i$ under $M_0: N(\mu_0, \sigma_0)$ and $M_1: N(\mu, \sigma_1)$. While \cite{berger1998bayes} originally used $\pi_1^N(\mu, \sigma_1) \propto \sigma_1^{-1}$, this lead to a marginal $m_1(y^*(\ell)) = (2|y_1^*-y_2^*|)^{-1}$ for a minimal sample $y^*(\ell)$. Such a form can result in indeterminate Bayes Factor bounds as the density approaches zero or infinity.

 To circumvent this issue and ensure well-defined bounds
  we instead adopt the non-informative priors $\pi_0^N(\sigma_0) = \sigma_0^{-1}$ and $\pi_1^N(\mu, \sigma_1) = \sigma_1^{-2}$ also known as the modified Jeffreys prior. This choice yields the following stable marginals for the minimal sample:
\begin{equation}
m_0^N(y^*(\ell)) = \frac{1}{2\pi((y_1^*)^{2}+(y_2^*)^{2})}, \quad m_1^N(y^*(\ell)) = \frac{1}{\sqrt{\pi(y_1^*-y_2^*)^{2}}}
\end{equation} 

and the Bayes factor is
\[
B_{01}\bigl(y^*(\ell)\bigr)
=
\frac{m_0\bigl(y^*(\ell)\bigr)}{m_1\bigl(y^*(\ell)\bigr)}
=
\frac{(y_1^*-y_2^*)^2}{2\sqrt{\pi}\,((y_1^*)^2+(y_2^*)^2)}.
\]
\[
\sup B_{01}\bigl(y^*(\ell)\bigr)
=
\frac{m_0\bigl(y^*(\ell)\bigr)}{m_1\bigl(y^*(\ell)\bigr)}
=
\frac{(y_1^*-y_2^*)^2}{2\sqrt{\pi}\,((y_1^*)^2+(y_2^*)^2)}
=
\frac{1}{\sqrt{\pi}}.
\]

\[
UB_T
=
B_{10}(y)\,\sup B_{01}\bigl(y^*(\ell)\bigr)
=
\frac{m_1(y)}{\sqrt{\pi}\,m_0(y)}.
\]
We can compute $m_0$ and $m_1$ as in previous example (see suplemental material) and obtain 
the  Bayes Factor upper bound and Lower Bound is given by
\[
 LB_T
=
\sqrt{\frac{n}{2}}
\left(\frac{\sum (y_i-\bar y)^2}{\sum (y_i-\mu_0)^2}\right)^{n/2}=\dfrac{1}{UB_T}.
\]

\subsection{Linear regression example}

\subsubsection{Two nested normal-linear models}
Consider the comparison of two nested normal-linear models $M_0 \subset M_1$, defined as in \cite{smith1980bayes} by
$ M_i: \textbf{y} \sim N(\textbf{A}_i\theta_i,\sigma^2\textbf{I}_n), \hspace{10pt} i = 0,1$
where $\textbf{A}_i$ is a full rank $p_i$ known matrix, $\textbf{y}$ is a vector with dimension $n$, $\boldsymbol{\theta_i} = (\theta_{i1},...,\theta_{ip_i})$ is a vector of $p_i$ unknown parameters and $\sigma^2$ is unknown. This can be written in matrix notation as
$$ {\displaystyle \textbf{y} = \textbf{A}_i\theta_i + \epsilon_i}, i = 0,1$$
where $\boldsymbol{\epsilon}_i \sim N(\textbf{0},\sigma^2 I_n)$ for $i = 0,1$. Let
$ \hat{\theta}_i = (A_i^T A_i)^{-1}A_i^T\textbf{y} \mbox { and } \textbf{R}_i = |\textbf{y} - A_i\hat{\theta}_i|^2$
denote the least squares estimator of $\theta_i$ and the residual sum of squares, respectively.
The Bayes factor in this case is given by
$$B_{01} = \frac{\int \int p(\textbf{y}|\textbf{A}_0,\boldsymbol{\theta}_0,\sigma)p(\boldsymbol{\theta}_0,\sigma|\textbf{A}_0)d\boldsymbol{\theta}_0 d\sigma}{\int \int p(\textbf{y}|\textbf{A}_1,\boldsymbol{\theta}_1,\sigma)p(\boldsymbol{\theta}_1,\sigma|\textbf{A}_1)d\boldsymbol{\theta}_1 d\sigma}$$
We define generalized prior $\pi(\theta_i,\sigma) \propto {\sigma}^{-(1+q_i)}$ where different values of q will result in different well-known priors 
 According to \cite{berger1996intrinsic}, $q_i = 0$ is the reference prior (example above), and $q_i = p_1 - p_0$ is the modified Jeffrey's prior. In fact in the Reference prior the boun fails. \\

Spiegelhalter and Smith \cite{spiegelhalter1982bayes} showed that for Jeffreys prior 
and they calculated the Bayes factor
\begin{equation}\label{generallmbf}
B_{01} = \frac{c_0}{c_1}[|\textbf{A}_1^T \textbf{A}_1|/|\textbf{A}_0^T \textbf{A}_0|]^{\frac{1}{2}}\big[1 + \frac{(p_1 - p_0)}{(n-p_1)}F\big]^{-(n/2)}
\end{equation}
where $F$ is the $F$-test statistics for comparing $M_0$ and $M_1$ which is clearly indeterminate due to the ratio of undefined constants. 
The $F$ in the Bayes Factor above is given by the following identity $\frac{\textbf{R}_0}{\textbf{R}_1} = 1+ \Big(\frac{p_1 - p_0}{n-p_1}\Big)F$\\
In \cite{spiegelhalter1982bayes} p.379, proposed a satisfactory solution to the problem of determining the ratio $c_0/c_1$ by introducing the concept of an imaginary training sample. Let $\textbf{A}_0(\ell)$ and $\textbf{A}_1(\ell)$ be the design matrix of $M_0$ and $M_1$ occurring in the "thought experiment" generating the imaginary training sample, they obtained the following result

\begin{equation}\label{constants_ratio}
\frac{c_0}{c_1} = [|\textbf{A}_1(\ell)^T\textbf{A}_1(\ell)|/|\textbf{A}_0(\ell)^T \textbf{A}_0(\ell)|]^{-\frac{1}{2}}
\end{equation} 
In general, the marginal distribution for $n$ samples and $i = 0,1$ is given by
$$ m_i(\textbf{y}) = \displaystyle\int_{-\infty}^{\infty}\int_{0}^{\infty}\frac{c_i}{\sigma^ {1+q_i}}(\frac{1}{\sqrt{2\pi}\sigma})^{n}\exp\{-\frac{1}{2\sigma^2}\Big[\textbf{R}_i(\textbf{y}) + (\theta_i - \textbf{y})^T A_i^T A_i (\theta_i - \textbf{y})\Big]\}d\theta_i d\sigma$$
Where $\textbf{R}_i(\textbf{y}) = (\textbf{y}- A_i\hat{\theta_i})^T(\textbf{y}- A_k\hat{\theta_i})$. See Supplementary material for detail calculations, the marginal can be simplified to 
$$ m_i(\textbf{y}) \propto c_i\sqrt{|\textbf{A}_i^{T} \textbf{A}_i|^{-1}\hspace{3pt}[4\pi]^{p_i}
\hspace{3pt}\textbf{R}_i(\textbf{y})^{-(n+q_i - p_i)}}\hspace{3pt}$$
and the Bayes factor is given by
$$ B_{01}^N(\textbf{y}) = \frac{m_0(\textbf{y})}{m_1(\textbf{y})} = \frac{c_0}{c_1}\sqrt{[4\pi]^{(p_0 - p_1)}\frac{|\textbf{A}_1^{T} \textbf{A}_1|}{|\textbf{A}_0^{T} \textbf{A}_0|} \frac{[\textbf{R}_1(\textbf{y})]^{n+q_1 - p_1}}{[\textbf{R}_0(\textbf{y})]^{n+q_0 - p_0}}}$$

Now we will compute the Bayes factor for the minimal training sample. Since we need at least one observation per unknown parameter, we can define the minimal training sample size as $n_{01}= \max\{p_0 + 1,p_1 + 1\}$
and the Bayes factor of the minimal training sample can be  written as
$$ B_{10}^N(\textbf{y}(\ell)) = \frac{m_1(\textbf{y}(\ell))}{m_0(\textbf{y}(\ell))} = \frac{c_1}{c_0}\sqrt{[4\pi]^{(p_1 - p_0)}\frac{|\textbf{A}_0^{T}(\ell) \textbf{A}_0(\ell)|}{|\textbf{A}_1^{T}(\ell) \textbf{A}_1(\ell)|}\frac{[\textbf{R}_0(\textbf{y}(\ell))]^{n_{01}+q_0 - p_0}}{[\textbf{R}_1(\textbf{y}(\ell))]^{n_{01}+q_1 - p_1}}}$$
Using Basic Lemma from \cite{berger1996intrinsic}
$$ B_{01}(\textbf{y}(-\ell)|\textbf{y}(\ell)) =  \sqrt{\frac{|\textbf{A}_1^{T} \textbf{A}_1|}{|\textbf{A}_0^{T} \textbf{A}_0|}\frac{|\textbf{A}_0^{T}(\ell) \textbf{A}_0(\ell)|}{|\textbf{A}_1^{T}(\ell) \textbf{A}_1(\ell)|}\frac{[\textbf{R}_1(\textbf{y})]^{n+q_1 - p_1}}{[\textbf{R}_0(\textbf{y})]^{n+q_0 + p_0}}
\frac{[\textbf{R}_0(\textbf{y}(\ell))]^{n_{01}+q_0 - p_0}}{[\textbf{R}_1(\textbf{y}(\ell))]^{n_{01} +q_1 - p_1}}}$$

and the General Empirical Bayes factor lower bound is given by

$$LB_E =  \sqrt{\frac{|\textbf{A}_1^{T} \textbf{A}_1|}{|\textbf{A}_0^{T} \textbf{A}_0|}\frac{[\textbf{R}_1(\textbf{y})]^{n+q_1 - p_1}}{[\textbf{R}_0(\textbf{y})]^{n+q - p_0}}}
\displaystyle \min_{\ell = 1,...,L}\sqrt{\frac{|\textbf{A}_0^{T}(\ell) \textbf{A}_0(\ell)|}{|\textbf{A}_1^{T}(\ell) \textbf{A}_1(\ell)|} \frac{[\textbf{R}_0(\textbf{y}(\ell))]^{n_{01} +q_0 -p_0}}{[\textbf{R}_1(\textbf{y}(\ell))]^{n_{01} +q_1 - p_1}}}$$
 also, the General Theoretical BF lower bound can be written as
\begin{equation}\label{BFGSS}
LB_T=  \sqrt{\frac{|\textbf{A}_1^{T} \textbf{A}_1|}{|\textbf{A}_0^{T} \textbf{A}_0|}\frac{[\textbf{R}_1(\textbf{y})]^{n+q_1 - p_1}}{[\textbf{R}_0(\textbf{y})]^{n+q_0 - p_0}}\inf_{y(\ell) \in D}\frac{|\textbf{A}_0^{T}(\ell) \textbf{A}_0(\ell)|}{|\textbf{A}_1^{T}(\ell) \textbf{A}_1(\ell)|}
\frac{[\textbf{R}_0(\textbf{y}(\ell))]^{n_{01} +q_0 - p_0}}{[\textbf{R}_1(\textbf{y}(\ell))]^{n_{01} +q_1 - p_1}}}
\end{equation} 

Following the dimensional argument in \cite{smith1980bayes}, we can significantly simplify the ratio of the determinants. First, note that the infimum operator is taken with respect to the responses of the imaginary training sample, $y(\ell)$. The determinant terms $|\textbf{A}_i^{T}(\ell) \textbf{A}_i(\ell)|$, however, depend strictly on the design matrices and are independent of $y(\ell)$. Therefore, they act as positive constants with respect to the infimum and can be factored out

\begin{equation*}
LB_T= \sqrt{\frac{|\textbf{A}_1^{T} \textbf{A}_1|}{|\textbf{A}_0^{T} \textbf{A}_0|}\frac{|\textbf{A}_0^{T}(\ell) \textbf{A}_0(\ell)|}{|\textbf{A}_1^{T}(\ell) \textbf{A}_1(\ell)|}\frac{[\textbf{R}_1(\textbf{y})]^{n+q_1 - p_1}}{[\textbf{R}_0(\textbf{y})]^{n+q_0 - p_0}}\inf_{y(\ell) \in D} \frac{[\textbf{R}_0(\textbf{y}(\ell))]^{n_{01} +q_0 - p_0}}{[\textbf{R}_1(\textbf{y}(\ell))]^{n_{01} +q_1 - p_1}}}
\end{equation*}

Now we can simplify the product of the determinants. Suppose that the full data set of size $n$ can be viewed approximately as an $(n/n_{01})$-fold replicate of the minimal training sample of size $n_{01}$. Under this assumption, the information matrix for the full sample grows proportionally to the sample size ratio
$$ \textbf{A}_i^T \textbf{A}_i \approx \left(\frac{n}{n_{01}}\right) \textbf{A}_i^T(\ell) \textbf{A}_i(\ell) $$

By applying the standard property of determinants for a $p_i \times p_i$ matrix, multiplying the matrix by a scalar $c = (n/n_{01})$ scales the determinant by $c^{p_i}$. Therefore
$$ |\textbf{A}_i^T \textbf{A}_i| \approx \left(\frac{n}{n_{01}}\right)^{p_i} |\textbf{A}_i^T(\ell) \textbf{A}_i(\ell)| $$

Taking the ratio of these determinants for $M_1$ versus $M_0$ yields
$$ \frac{|\textbf{A}_1^T \textbf{A}_1|}{|\textbf{A}_0^T \textbf{A}_0|} \approx \left(\frac{n}{n_{01}}\right)^{p_1 - p_0} \frac{|\textbf{A}_1^T(\ell) \textbf{A}_1(\ell)|}{|\textbf{A}_0^T(\ell) \textbf{A}_0(\ell)|} $$

Substituting this result back into our factored determinant product:
$$ \left[ \left(\frac{n}{n_{01}}\right)^{p_1 - p_0} \frac{|\textbf{A}_1^T(\ell) \textbf{A}_1(\ell)|}{|\textbf{A}_0^T(\ell) \textbf{A}_0(\ell)|} \right] \frac{|\textbf{A}_0^T(\ell) \textbf{A}_0(\ell)|}{|\textbf{A}_1^T(\ell) \textbf{A}_1(\ell)|} = \left(\frac{n}{n_{01}}\right)^{p_1 - p_0} $$
Notice that the determinants of the training samples cross cancel completely.\\
Consequently, the General Theoretical BF lower bound can be rewritten without the need to compute the determinants of the design matrices for the minimal training samples. The expression simplifies to depend only on the sample sizes, parameter dimensions, and the residual sums of squares:

\begin{equation}\label{BFGSS_simplified}
LB_T \approx \left(\frac{n}{n_{01}}\right)^{\frac{p_1 - p_0}{2}} \sqrt{ \frac{[\textbf{R}_1(\textbf{y})]^{n+q_1 - p_1}}{[\textbf{R}_0(\textbf{y})]^{n+q_0 - p_0}} \inf_{y(\ell) \in D} \frac{[\textbf{R}_0(\textbf{y}(\ell))]^{n_{01} +q_0 - p_0}}{[\textbf{R}_1(\textbf{y}(\ell))]^{n_{01} +q_1 - p_1}} }
\end{equation}

\subsubsection{The One-way Layout (ANOVA)}

Assume that we have $m$ groups of observations with $n_i$ observations in the i-th group where
$$ y_{ij} \sim N(\mu_i,\sigma^2), i = 1,...,m, j = 1,...,n_i.$$
independently, given $\mu_1,...,\mu_m,\sigma^2$. Consider the following models
$$ M_0: \mu_1 = ... = \mu_m \mbox{ vs } M_1: \mu_i \neq \mu_j \mbox{ for some } i \neq j. $$

In this case, $p_1 = m, p_0 = 1,n = \sum_{i=1}^{m} n_i$ and is not difficult to write the general form of the matrices $\textbf{A}_0$ and $\textbf{A}_1$
where $A_0$ and $A_1$ are ($n \times m$) matrices and each block of ones in $A_1$ corresponds to each of the m groups, hence every block i has $n_i$ rows. Therefore,  $\det(\textbf{A}_0^T \textbf{A}_0) = n$, \\ and 

\begin{equation}\label{determinantM1}
\det(\textbf{A}_1^T \textbf{A}_1) = \prod_{i=1}^{m}n_i
\end{equation}
since $\textbf{A}_1^T \textbf{A}_1$ is diagonal matrix
\\
Therefore, we can easily calculate the term $|\textbf{A}_1^T \textbf{A}_1|/|\textbf{A}_0^T \textbf{A}_0|$ in equation \eqref{generallmbf} as
\begin{equation}\label{matrixdetermianants}
|\textbf{A}_1^T \textbf{A}_1|/|\textbf{A}_0^T \textbf{A}_0| = \frac{\prod_{i=1}^{n}n_i}{n}
\end{equation}
In this case, the minimal training sample requires at least one observation in each group plus one extra observation in any of the groups to be able to estimate $\sigma^2$. Therefore we require that $n_i = 1, i=1,...,j-1,j+1,...,n$ and $n_j = 2$ for some $j \in [1,m]$.

For the case mentioned above, we can use equation \eqref{constants_ratio}, and the result obtained in \eqref{determinantM1} to get $c_1$ as
$det(\textbf{A}_1(\ell)^T \textbf{A}_1(\ell)) = (1\cdot\cdot\cdot1\cdot2\cdot1\cdot\cdot\cdot1) = 2$ . Similarly, we can obtain $c_0$ as $\det(\textbf{A}_0(\ell)^T \textbf{A}_0(\ell)) = m + 1$

Hence, using this results we can calculate the expression in equation \eqref{constants_ratio} 
\begin{equation}\label{constants_ratio_mts}
\frac{c_0}{c_1} = \big(\frac{\det(\textbf{A}_0(\ell)^T \textbf{A}_0(\ell))}{\det(\textbf{A}_1(\ell)^T \textbf{A}_1(\ell))}\big)^{\frac{1}{2}} =  \big( \frac{m+1}{2} \big)^{\frac{1}{2}}
\end{equation}
Now we will compute the $LB_{01}$ in equation $\eqref{BFGSS}$ for ANOVA models. Since the $A_0(\ell)$ and $A_1(\ell)$ do not depend on the observations for ANOVA models, and the square root is a monotonic function, we can rewrite the expression as 
$$ \displaystyle \sup_{y(\ell) \in D_n}\sqrt{\frac{|\textbf{A}_0^{T}(\ell) \textbf{A}_0(\ell)|}{|\textbf{A}_1^{T}(\ell) \textbf{A}_1(\ell)|}\frac{[\textbf{R}_0(\textbf{y}(\ell))]^{n_{01} +q_i - p_0}}{[\textbf{R}_1(\textbf{y}(\ell))]^{n_{01} +q - p_1}}}  = \sqrt{\frac{|\textbf{A}_0^{T}(\ell) \textbf{A}_0(\ell)|}{|\textbf{A}_1^{T}(\ell) \textbf{A}_1(\ell)|} \displaystyle \sup_{y(\ell) \in D_n}\frac{[\textbf{R}_0(\textbf{y}(\ell))]^{n_{01} +q_i - p_0}}{[\textbf{R}_1(\textbf{y}(\ell))]^{n_{01} +q_i - p_1}}}
$$
and the General Theoretical Bound for ANOVA can be written as 
$$ UB_T =  \sqrt{\frac{|\textbf{A}_1^{T}(\ell) \textbf{A}_1(\ell)|}{|\textbf{A}_0^{T}(\ell) \textbf{A}_0(\ell)|}\frac{|\textbf{A}_0^{T} \textbf{A}_0|}{|\textbf{A}_1^{T} \textbf{A}_1|}\frac{[\textbf{R}_0(\textbf{y})]^{n+q_0 - p_0}}{[\textbf{R}_1(\textbf{y})]^{n+q_1 - p_1}}\sup_{y(\ell) \in D}
\frac{[\textbf{R}_1(\textbf{y}(\ell))]^{n_{01} +q_1 - p_1}}{[\textbf{R}_0(\textbf{y}(\ell))]^{n_{01} +q_0 - p_0}}}
$$  
 Similarly we can also get the bound for the empirical case. 
\subsubsection{ANOVA: Full Jeffrey's  $q_k = p_k,$, for $k = i,j$}
We have calculated above the General Empirical SS Bayes factor for ANOVA. Following the approach of\cite{spiegelhalter1982bayes}, we can choose the full Jeffrey's prior $q_i = p_i$, and using an imaginary training sample that provides the maximum support for $M_0$ which results in $F{_y(\ell)} = 0$, we obtain that the General Theoretical SS BF for ANOVA under the Full Jeffrey's prior
$$LB_E = \sqrt{\frac{m+1}{2}\big[\displaystyle\prod_{i=1}^{m}n_i\big/n\big]\Big(\frac{\textbf{R}_1(\textbf{y})}{\textbf{R}_0(\textbf{y})}\Big)^{n}\min_{\ell = 1,...,L}
\Big(\frac{\textbf{R}_0(\textbf{y}(\ell))}{\textbf{R}_1(\textbf{y}(\ell))}\Big)^{m+1}}
$$
Now expressing the ratio of the residual sum of squares as an F-statistic we obtain
$$LB_E=\sqrt{\frac{m+1}{2}\big[\displaystyle\prod_{i=1}^{m}n_i\big/n\big]\Big(1 + \frac{(m - 1)}{(n - m)}F_{y}\Big)^{-n}\min_{\ell = 1,...,L}
\Big(1 + \frac{(p_1 - p_0)}{(n_{01} - p_1)}F_{y(\ell)}\Big)^{m+1}}$$

The General Theoretical Bayes factor bound for ANOVA is given by
$$UB_T=\sqrt{\frac{2}{m+1}\big[\displaystyle n\big/\prod_{i=1}^{m}n_i\big]\Big(1 + \frac{(m - 1)}{(n - m)}F_{\textbf{y}}\Big)^{n}\sup_{y(\ell) \in D}
\Big[\frac{\textbf{R}_1(\textbf{y}(\ell))}{\textbf{R}_0(\textbf{y}(\ell))}\Big]^{m+1}}$$

Note that $\textbf{R}_0(y(\ell)) \geq \textbf{R}_1(y(\ell))$, then following our approach we have that 
$$ \sup_{y(\ell) \in D}
\frac{[\textbf{R}_1(\textbf{y}(\ell))]}{[\textbf{R}_0(\textbf{y}(\ell))]} = 1 \Longrightarrow\sup_{y(\ell) \in D}
\Big[\frac{\textbf{R}_1(\textbf{y}(\ell))}{\textbf{R}_0(\textbf{y}(\ell))}\Big]^{m+1} = 1, \hspace{5pt}, \forall m > 0 $$ 
and the General Theoretical Bayes Factor Bound under the full Jeffrey's prior for ANOVA models becomes
$$ UB_T=\sqrt{\frac{2}{m+1}\big[\displaystyle n\big/\prod_{i=1}^{m}n_i\big]\Big(1 + \frac{(m - 1)}{(n - m)}F_{\textbf{y}}\Big)^{n}} = $$

$$ UB_T =  \sqrt{\frac{2}{m+1}\big[\displaystyle n\big/\prod_{i=1}^{m}n_i\big] \Big[\frac{\textbf{R}_0(\textbf{y})}{\textbf{R}_1(\textbf{y})}\Big]^n}
$$ 
\subsubsection{ANOVA: Modified Jeffrey's Prior ($q_j = p_j - p_i, q_i = 0$)} The General Empirical SS Bayes Factor for ANOVA models under the Modified Jeffrey's prior is given by 
$$ UB_E = \sqrt{\frac{2}{m+1}\big[n\big/\displaystyle\prod_{i=1}^{m}n_i\big]\Big(\frac{\textbf{R}_0(\textbf{y})}{\textbf{R}_1(\textbf{y})}\Big)^{n-1}\max_{\ell = 1,...,L}
\Big(\frac{\textbf{R}_1(\textbf{y}(\ell))}{\textbf{R}_0(\textbf{y}(\ell))}\Big)^{m}}
$$ 
The General Theoretical SS Bayes Factor for ANOVA models under the Modified Jeffrey's prior is given by 
$$ UB_T = \sqrt{\frac{2}{m+1}\big[n\big/\displaystyle\prod_{i=1}^{m}n_i\big]\Big(\frac{\textbf{R}_0(\textbf{y})}{\textbf{R}_1(\textbf{y})}\Big)^{n-1}\sup_{y(\ell) \in D}
\Big(\frac{\textbf{R}_1(\textbf{y}(\ell))}{\textbf{R}_0(\textbf{y}(\ell))}\Big)^{m}}
$$ 
$$ UB_T = \sqrt{\frac{2}{m+1}\big[n\big/\displaystyle\prod_{i=1}^{m}n_i\big]\Big(1 + \frac{(m - 1)}{(n - m)}F_{y}\Big)^{n-1}} $$ 
\subsubsection{Conclusions from the ANOVA case}
\begin{enumerate}
    \item For the Full Jeffreys and Modified Jeffreys the SS bounds are informative useful and close to each other. The point is that under both priors the BF can be expressed as a function of the F-Statistics. 
    \item However for the Reference Prior ($q_k = 0, k = i,j$), it can be shown that  bound fails, as it is uninformative. It is curious that the bound is so sensitive to the initial objective prior.
    \item This adds to the conclusion of Berger and Pericchi (1996) and elsewhere, that the Modified Jeffreys prior is the better choice, at least in the Linear Gaussian Model, and perhaps much more broadly.
    \item Perhaps it would be interesting to explore to which extend the Modified Jeffreys prior, has additional properties of matching, besides simplicity, and one-one relationship with F-Statistics.
\end{enumerate}
\section{Least Favorable Priors}\label{sec3}
Section 4.1 will present the ideas of what are least favorable prior (LF priors) and least favorable Bayes factors (LF Bayes Factors). Following with the methodology we continue with the examples already presented in section 3 and present some results within those examples. 
\subsection{Methodology} \label{subsec1}
We have discussed in previous sections the concept of obtaining a bound for the intrinsic Bayes factor. In the process it is shown that we have to compute a extrema on the correction factor that is 
$\{ \arg\sup_{y(\ell)}\ B_{01}^N\!\big(y(\ell)\big)\}$ or $\{ \arg\inf_{y(\ell)}\ B_{10}^N\!\big(y(\ell)\big)\}$  we denote the training sample that produces this extrema $y^s(\ell)$. The training sample \(y^s(\ell)\) may be empirical or theoretical, depending on whether the extremum is taken over observed training samples or over the full training-sample space. For simplicity in notation from now on in our examples all trainings samples are theoretical We define the LF prior as the prior defined on $y^s(\ell)$, more specifically

\begin{equation}
\pi_k^{\mathrm{LF}}(\theta_k)
\;:=\;\pi_k^N\!\big(\theta_k \mid y^s(\ell)\big)
\;=\;\frac{f_k\!\big(y^s(\ell)\mid \theta_k\big)\,\pi_k^N(\theta_k)}
            {m_k^N\!\big(y^s(\ell)\big)}
\end{equation}
Using this $\pi_k^{\mathrm{LF}}(\theta_k)$ we define the LF Bayes factor 

\begin{equation}
B_{01}^{\mathrm{LF}}(y)
\;=\;\frac{m_0^{\mathrm{LF}}(y)}{m_1^{\mathrm{LF}}(y)}
\end{equation}
where $m_k^{LF}$ are just the marginals under $\pi_k^{\mathrm{LF}}(\theta_k)$.

\subsection{  Normal precision Example (Continuation)}
Recall that the bound is attained at $\hat{d}=\sqrt{2/h_0}$, so we take training samples of size two i.e $y(l)=(y_1,y_2)$ that satisfy this condition and compute the LF prior. We first compute the joint on the training sample
\begin{equation}
\pi(\mu,h|y(\ell)) = \frac{f(y(\ell)|\mu,h)\pi(\mu)\pi(h)}{\int f(y(\ell)|\mu,h)\pi(\mu)\pi(h)d\mu dh} =  \frac{\frac{1}{h}\frac{h}{2\pi}\exp\{-\frac{h}{2}[(y(1) - \mu)^2 + (y(2) - \mu)^2]\}}{\frac{1}{2\pi}\displaystyle\int  \exp\{-\frac{h}{2}(\frac{d(\ell)^2}{2} + 2(\mu - \bar{y}(\ell))^2 \}d\mu d h} 
\end{equation}
Hence,
\begin{equation}
\pi(\mu,h|d=\hat{d}) = \frac{\exp\{-\frac{h}{2}[\frac{\hat{d}^2}{2} + 2(\mu - \bar{y}(\ell))^2] \}\hat{d}}
{2\pi}
\end{equation}
Integrating with respect to $\mu$ we obtain the prior for $h$,
\\
$$  \pi^{LF}(h|d = \sqrt{\frac{2}{h_0}}) = \frac{\sqrt{\frac{2}{h_0}}\exp\{\frac{-h(\sqrt{\frac{2}{h_0}})^2}{4}\}}{\sqrt{2\pi}\sqrt{2h}} = \frac{1}{\sqrt{2\pi hh_0}}\exp\{-\frac{h}{2h_0}\}$$
Hence, $\pi^{LF}(h)$ is $Gamma(\alpha = 1/2,\beta = 2h_0)$.

$$ \pi^{LF}(\mu|d = \sqrt{\frac{2}{h_0}}) = \frac{8\pi}{\frac{2}{h_0} + 4(\mu - \bar{y}(\ell))^2} = \frac{2\pi}{\frac{1}{2h_0} + (\mu - \bar{y}(\ell))^2} = \frac{4h_0\pi}{1 + (\frac{\mu - \bar{y}(\ell)}{\sqrt{2h_0}})^2} $$
The distribution of $\mu$ for $B_{01}$ is proportional to a $Cauchy(\bar{y}(\ell),\gamma)$ where $\gamma = \sqrt{2h_0}$. 
$$  \pi^{LF}(\mu) \propto \frac{1}{\pi \gamma\Big(1 + (\frac{\mu - \bar{y}(\ell)}{\gamma})^2\Big)}  $$
Now as a comparison consider the intrinsic priors for the same hypothesis test above, 
$$  \pi^{I}(h) \propto \frac{1}{\pi}\frac{\sqrt{\frac{h}{h_0}}}{\frac{h+1}{h_0}}\frac{1}{h} = \frac{1}{\pi h_0}\frac{(\frac{h}{h_0})^{1/2 -1}}{(\frac{h}{h_0} + 1) }$$

Hence, $\pi^{I}(h)$ is  $SBeta2(h|p=1/2,q=1/2,b=h_0)$, which is a Scale beta2.\\
\textbf{Comment}: The IPrior for $\mu$ is improper but well-calibrated (see \cite{pericchi2005model}). On the other hand, the $LF$ prior for $\mu$ is proper for both hypotheses, and not the same. 
In fact it looks that the prior on $\mu|H_0$ is Normal, with precision $2\cdot h_0$ and mean $\bar{y}^s(\ell)$ \\
\\
With the priors above, we can compute the marginals and obtain the Least Favorable Bayes Factor
\begin{equation}
B_{01}^{\mathrm{LF}}(y)
= \frac{m_0^{\mathrm{LF}}(y)}{m_1^{\mathrm{LF}}(y)}
= \frac{\sqrt{\pi}\,h_0^{(n+1)/2}}{2^{n/2}\,\Gamma\!\big(\tfrac{n+1}{2}\big)} \;
\frac{\Big(S_n^2+\dfrac{1}{h_0}+\dfrac{2n}{n+2}(\bar y-\bar y^s(\ell))^2\Big)^{(n+1)/2}}
{\exp\!\Big\{\dfrac{h_0}{2}\Big[S_n^2+\dfrac{2n}{n+2}(\bar y-\bar y^s(\ell))^2\Big]\Big\}}
\end{equation}
\textbf{Comment}:Expression (21) depends on the choice of a particular training sample that satisfies the condition $\hat{d}=\sqrt{2/h_0}$, which defines an "orbit" of training samples. The natural choice of a point in that orbit is the "centered"  $\bar y^s(\ell)=\bar{y}$. This choice makes $B^{LF}_{01}$ close to LB (see section 4), it is quite robust with respect to small changes of $y^s(\ell)$ and leads to the following natural simplification,
\begin{equation*}
B_{01}^{\mathrm{LF}}(y)
= \frac{\sqrt{\pi}\,h_0^{(n+1)/2}}{2^{n/2}\,\Gamma\!\big(\tfrac{n+1}{2}\big)}
\left(S_n^2+\frac{1}{h_0}\right)^{(n+1)/2}
\exp\!\left(-\frac{h_0}{2}S_n^2\right).
\end{equation*}
Also, if you take $\displaystyle\lim_{n\to \infty}\dfrac{LB}{B_{01}^{LF}}\to 1$ under $H_0$ which implies that both are close under $H_0$ if you choose $\bar y^s(\ell)=\bar{y}$.

\subsection{Normal Mean hypothesis test (Continuation) }
From previous example we know $y^s(\ell)=\mu_0$.

We now proceed to compute the LF prior. First under the null the parameter space is empty so one may write $\pi_0^{\mathrm{LF}}$ as the unit mass on the
empty parameter while $
m_0^N\!\big(y^s(\ell)\big)=f_0\!\big(y^s(\ell)\big)
=\frac{1}{\sqrt{2\pi}\,\sigma_0}.$
 Under $H_1$ 
with $y^s(\ell)=\mu_0$ and $\pi_1^N(\mu)=1$ we get 
\[
\pi_1^{\mathrm{LF}}(\mu)
=\frac{\phi\!\left(\frac{\mu_0-\mu}{\sigma_0}\right)}
       {\int_{-\infty}^{\infty}\phi\!\left(\frac{\mu_0-\mu}{\sigma_0}\right)\,d\mu}
=\frac{1}{\sqrt{2\pi}\,\sigma_0}\exp\!\left\{-\frac{(\mu-\mu_0)^2}{2\sigma_0^2}\right\}.
\]
It can be clearly seen that the LF prior for $\mu$ follows a normal distribution with mean $\mu_0$ and variance $\sigma_0^2$. Now we procede to compute the LF Bayes factor. Let $\Delta=\bar{y}-\mu_0$ and $\mu^*=\dfrac{n\bar y+\mu_0}{n+1}$
\begin{equation*}
\begin{split}
m_1^{\mathrm{LF}}(y)
& =\int_{-\infty}^{\infty} f_1(y\mid\mu)\,\pi_1^{\mathrm{LF}}(\mu)\,d\mu \\
& = 2\pi\sigma_0^2)^{-(n+1)/2}
\int_{-\infty}^{\infty}
\exp\!\left\{-\frac{1}{2\sigma_0^2}\Big(S_1^2+n(\bar y-\mu)^2+(\mu-\mu_0)^2\Big)\right\}d\mu. \\
&=(2\pi\sigma_0^2)^{-(n+1)/2}
\exp\!\left\{-\frac{1}{2\sigma_0^2}\Big(S_1^2+\frac{n}{n+1}\Delta^2\Big)\right\}
\int_{-\infty}^{\infty}
\exp\!\left\{-\frac{n+1}{2\sigma_0^2}(\mu-\mu^*)^2\right\}d\mu. \\
& =(2\pi\sigma_0^2)^{-n/2}\,(n+1)^{-1/2}\,
\exp\!\left\{-\frac{S_1^2}{2\sigma_0^2}
-\frac{n}{2\sigma_0^2(n+1)}\,\Delta^2\right\}
\end{split}
\end{equation*}
Hence 
\begin{equation}
B_{10}^{\mathrm{LF}}(y)
=\frac{m_1^{\mathrm{LF}}(y)}{m_0(y)}
=\frac{1}{\sqrt{\,n+1\,}}\,
\exp\!\left\{\frac{n^2}{2\sigma_0^2(n+1)}\,\Delta^2\right\}
\end{equation}
Equivalently, with $z=\sqrt{n}\,\Delta/\sigma_0$,
\[
B_{01}^{\mathrm{LF}}(z)=\sqrt{\,n+1\,}\,
\exp\!\left\{-\frac{n}{n+1}\,\frac{z^2}{2}\right\}.
\]
\subsubsection{Intrinsic Bayes Factor vs Least Favorable Prior}
From \cite{berger1996intrinsic},
The intrinsic Bayes factor in favor of \(H_0:\mu=\mu_0\) versus
\(H_1:\mu\neq\mu_0\) is
\[
\mathrm{B}_{01}^{IP}(z)
=
\sqrt{2n+1}
\exp\!\left(
-\frac{2n}{2n+1}\frac{z^2}{2}
\right).
\]
The next plot graphs both Bayes Factors. It can be notice that both show similar asymptotic behavior. The next proposition shows a relationship between both, we would add the prove to the suplementary material.  

\begin{figure}[htbp]
     \centering
     \begin{subfigure}[b]{0.48\textwidth}
         \centering
         \includegraphics[width=\textwidth]{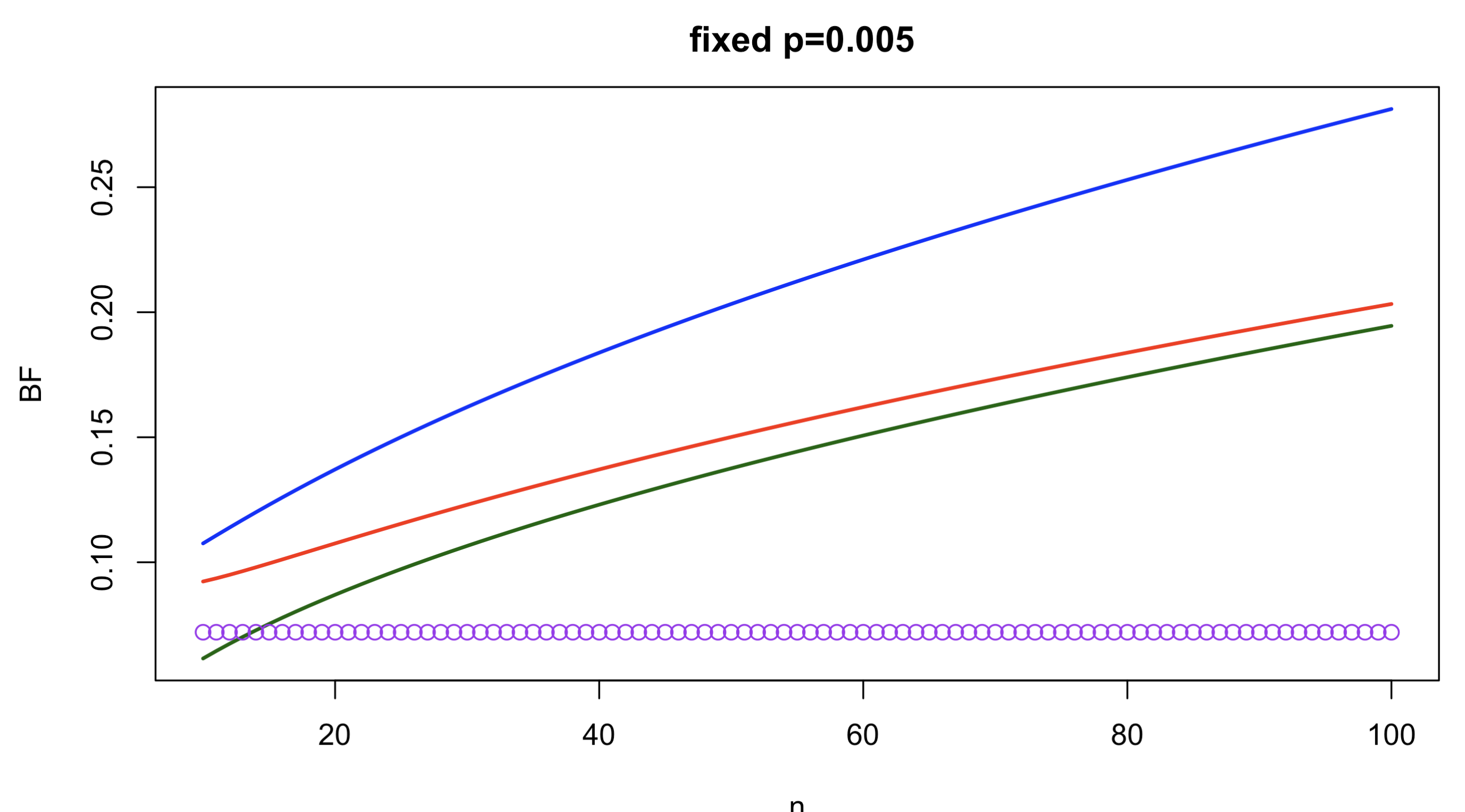}
         \caption{Fixed $p$, changing $n$.}
         \label{fig:known_var_n}
     \end{subfigure}
     \hfill
     \begin{subfigure}[b]{0.48\textwidth}
         \centering
         \includegraphics[width=\textwidth]{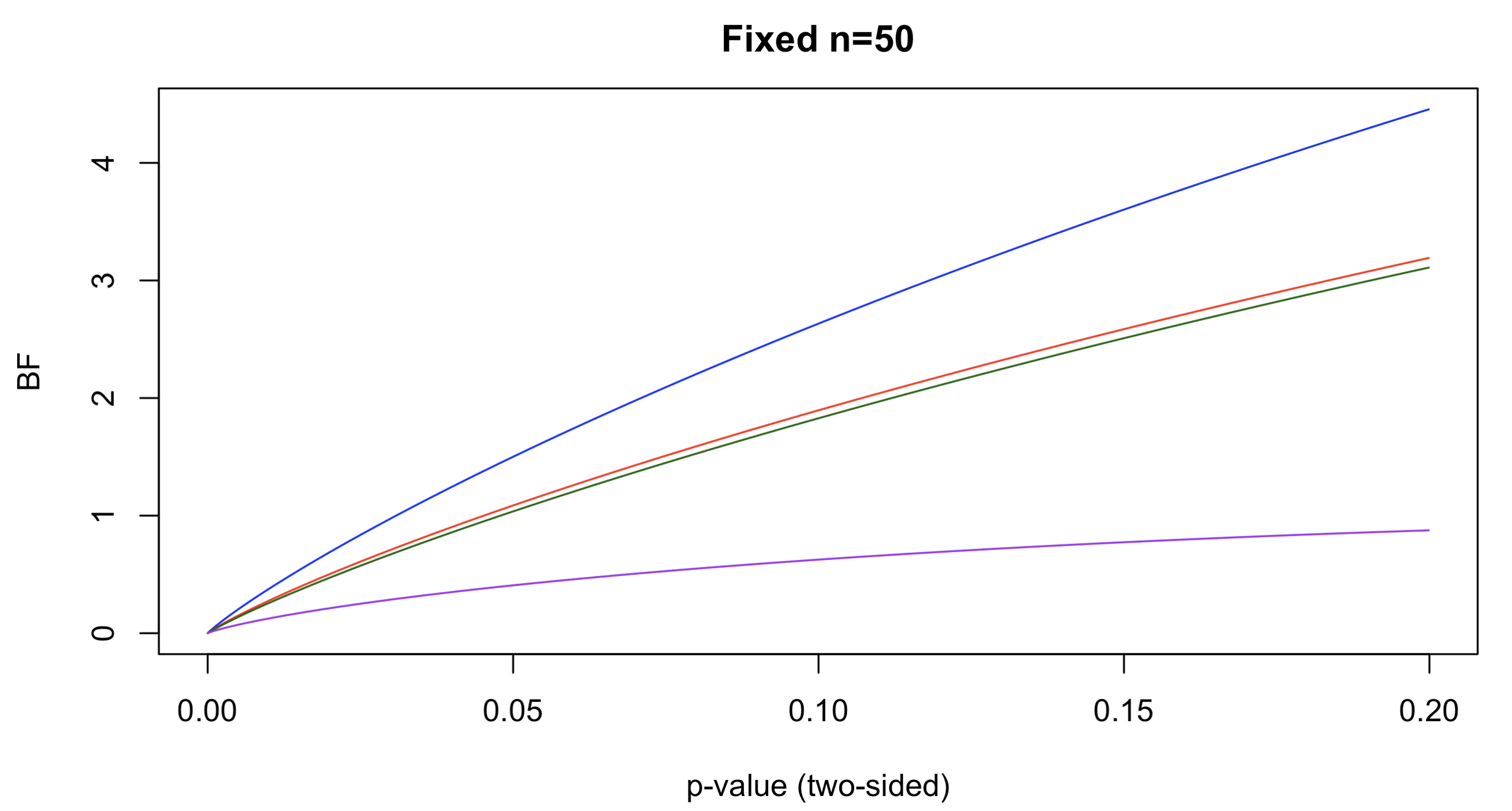}
         \caption{Fixed $n$, changing $p$.}
         \label{fig:known_var_p}
     \end{subfigure}
     
     \caption{Comparison of Lower Bounds and Bayes Factors: $-ep\log(p)$ (purple), $LB$ (green), $B_{01}^{LF}$ (red), and the intrinsic Bayes Factor (blue).}
     \label{fig:known_variance_comparison}
\end{figure}
\begin{proposition}

For fixed \(p\)-value 
\[
\frac{B_{01}^{\mathrm{IP}}(z)}{B_{01}^{\mathrm{LF}}(z)}
\longrightarrow
\sqrt{2},
\qquad n\to\infty.
\]
\end{proposition}
\subsection{Normal Mean hypothesis test  with $\sigma_0$ unknown example (continuation) }
 With the priors mentioned above we can obtain meaningful marginals 
and the Bayes factor on $y(\ell)$ is 
\[
B_{01}^N\!\big(y(\ell)\big)
=\frac{m_0(y(\ell))}{m_1(y(\ell))}
=\frac{y_1-y_2}{2\sqrt{\pi}\,(y_1^2+y_2^2)}
=\frac{D^2}{2\sqrt{\pi}\,(y_1^2+y_2^2)}
\]
To compute the supremum on the correction factor we apply Cauchy-Swartz inequality and obtain 
\[
\sup_{y(\ell)} B_{01}^N\!\big(y(\ell)\big)=\frac{1}{\sqrt{\pi}},
\qquad\text{attained at }y_2=-y_1
\]
where $y^s(\ell)=(y_1,y_2=-y_1)$
\\
No we proceed to compute the LF priors. We center at $\mu_0$ and take the arg–sup set $y_2=-y_1+2\mu_0$,
so $\bar y(\ell)=\mu_0$ and $\sum_{i=1}^{\ell}(y_i-\mu_0)^2 \;=\; \frac{D^2}{2}$

Under $H_0$, 
\[
\pi_0^{\mathrm{LF}}(\sigma)
=\frac{f_0(y(\ell)\mid\sigma)\,\pi_0^N(\sigma)}{m_0^N(y(\ell))}
=\dfrac{D^2}{2}\sigma^{-3}\exp\!\Big(-\frac{D^2}{4\sigma^2}\Big)
\]
If we let $\tau=\sigma^2$, then $\tau \sim \text{Inv-Gamma}(1,D^2/4)$
\\
Under $H_1$,
\[
 \pi_1^{\mathrm{LF}}(\mu,\sigma)
=\frac{|D|}{\pi}\ \sigma^{-3}\exp\!\left\{-\frac{D^2}{4\sigma^{2}}-\frac{(\mu-\mu_0)^2}{\sigma^{2}}\right\} 
\]
\\
Now we compute the LF Bayes Factor.
Computing the marginals we obtain, 
\[
m_0^{\mathrm{LF}}(y)
= \frac{D^2}{2}\,(2\pi)^{-n/2}\,\frac{1}{2}\,\Gamma\!\Big(\frac{n+2}{2}\Big)\, (\tfrac{S_0^2}{2}+\tfrac{D^2}{4})^{-(n+2)/2}
\]
\[
 m_1^{\mathrm{LF}}(y)
= \frac{|D|}{\sqrt{\pi}}\,\frac{\sqrt{2}}{\sqrt{n+2}}\,(2\pi)^{-n/2}\,
\frac{1}{2}\,\Gamma\!\Big(\frac{n+1}{2}\Big)\, \bigg(\frac12\Big(S_1^2+\frac{2n}{n+2}(\bar y-\mu_0)^2\Big)+\frac{D^2}{4}\bigg)^{-(n+1)/2}
\]
The computation of $m_0^{LF}$ would be added it to the suplementary material. 
Using $S_0^2=S_1^2+n(\bar y -\mu_0)^2$ and Gamma identities, the $D$–terms cancel and
\begin{equation}
 B_{10}^{\mathrm{LF}}(y)=\frac{m_1^{\mathrm{LF}}(y)}{m_0^{\mathrm{LF}}(y)}
=\sqrt{\frac{2}{n}}\left(\frac{S_0^2}{S_1^2}\right)^{\!n/2}
\end{equation}

\subsubsection{LF Bayes Factor vs.\ Intrinsic Bayes Factor}
Expression (19) can be written in terms of the p-value using the $t$ distribution 
\[
B_{01}^{\mathrm{LF}}(p,n)
=
\sqrt{\frac{n}{2}}\,(1+x)^{-n/2},\quad x=\frac{t^2}{n-1}\ge 0,
\qquad
t=t^{-1}_{n-1}(1-p/2).
\]
while the intrinsic Bayes factor of \cite{berger1996intrinsic} is
\[
B_{01}^{\mathrm{IP}}(p,n)
\approx
\sqrt{2n}\,(1+x)^{-n/2}\,
\frac{x}{1-e^{-x}}.
\]
The following proposition shows a relationship betwen both Bayes Factors. Proof is avaible in the suplementary material. 
\begin{proposition}
For the normal mean testing problem with unknown variance,
the intrinsic and least favorable Bayes factors satisfy
\[
B_{01}^{\mathrm{IP}}(p,n)\;\ge\;2\,B_{01}^{\mathrm{LF}}(p,n).
\]
Moreover, equality holds only in the limit \(x\to 0\) (equivalently, \(t\to 0\)).
\end{proposition}

Graph below shows comparison between $LF$ Bayes Factor, Intrinsic Bayes Factor and the $-eplogp$ bound. 

\begin{figure}[h]
     \centering
     \begin{subfigure}[b]{0.40\textwidth}
         \centering
         \includegraphics[width=\textwidth]{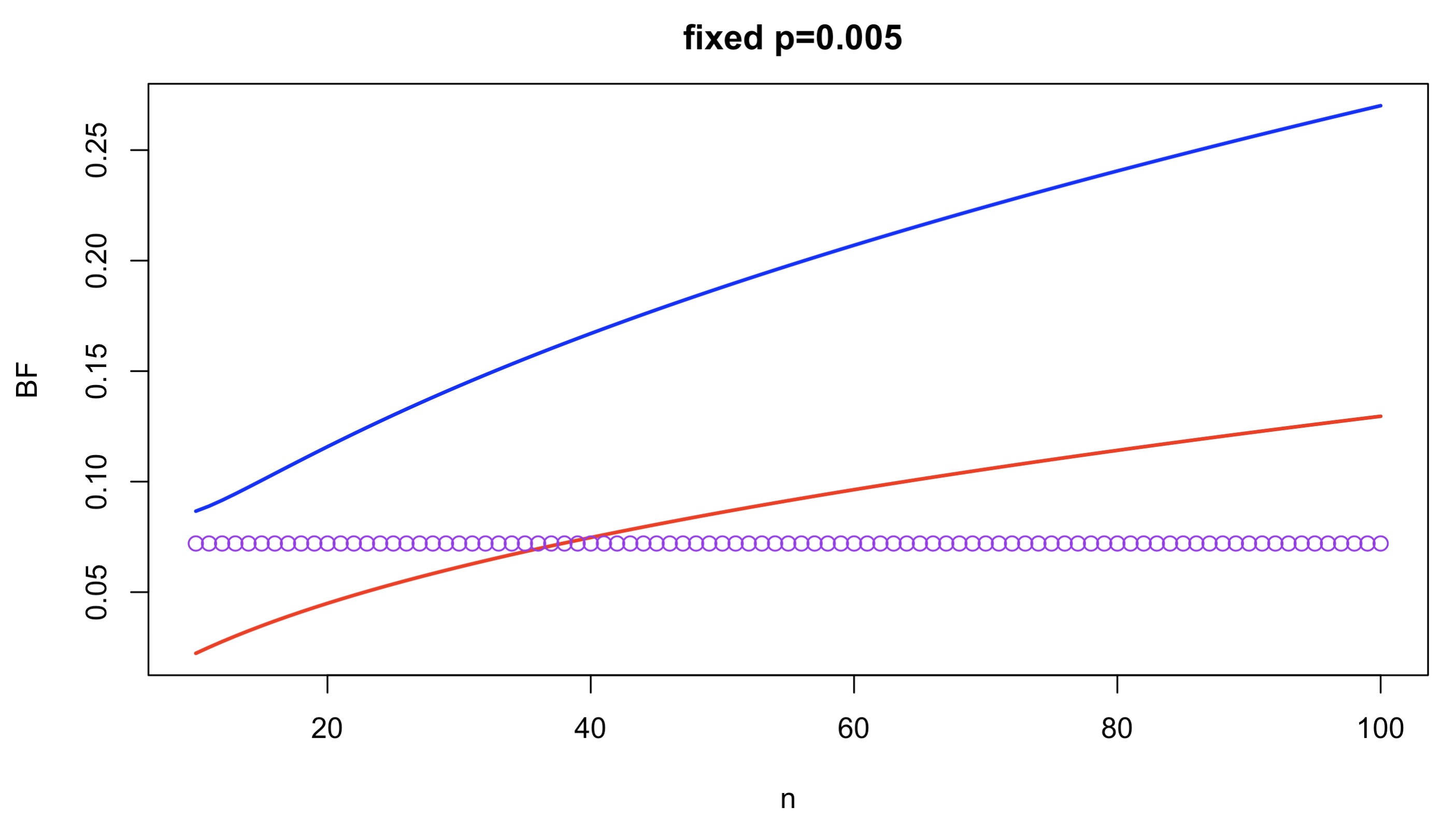}
         \caption{Fixed $p$, changing $n$.}
         \label{fig:graph1}
     \end{subfigure}
     \hfill 
     \begin{subfigure}[b]{0.40\textwidth}
         \centering
         \includegraphics[width=\textwidth]{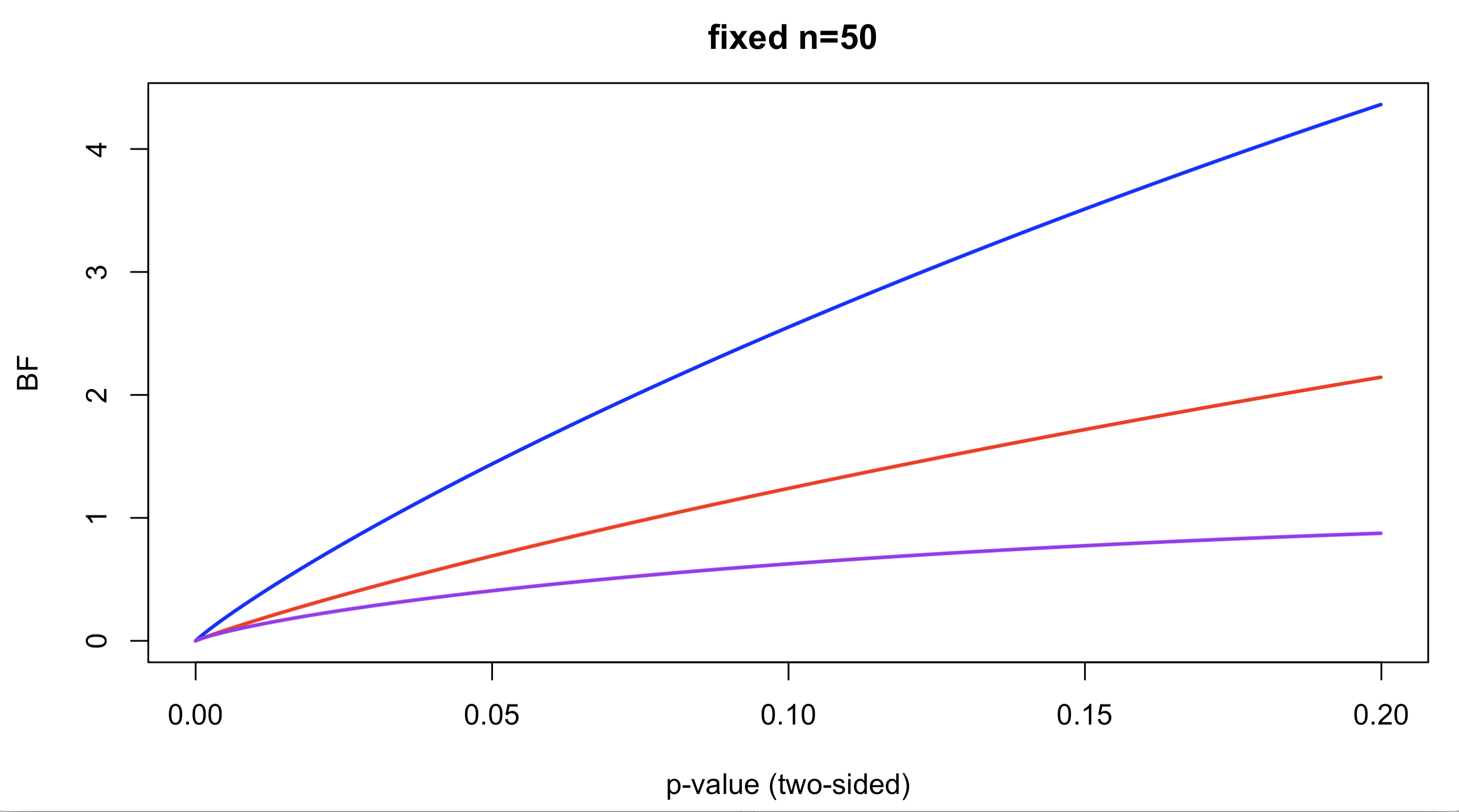}
         \caption{Fixed $n$, changing $p$.}
         \label{fig:graph2}
     \end{subfigure}
     
     \caption{Lower Bounds on the Bayes Factors: $-ep\log(p)$ (purple), $B_{01}^{LF}$ (red), and the intrinsic Bayes Factor (blue).}
     \label{fig:total_comparison}
\end{figure}
The next result shows that, for the normal mean test with unknown variance,
the least favorable Bayes factor reproduces \cite{schwarz1978estimating}  BIC approximation up to a
constant multiplicative factor. The results follows from the fact that BIC Bayes factor is given by $B_{01}^{\mathrm{BIC}}
\approx
\sqrt{n}\,
\Big(\frac{S_1^2}{S_0^2}\Big)^{n/2}$
\begin{proposition}
For the normal mean testing problem with unknown variance
the Schwarz approximation and the least favorable Bayes factor satisfy
\[
B_{01}^{\mathrm{BIC}}
\;\approx\;
\sqrt{2}\,B_{01}^{\mathrm{LF}}.
\]
\end{proposition}

\subsection{Application to a Real Data set}
We consider the classical sleep data analyzed by W.~S.~Gosset (``Student''), originally taken from \cite{cushny1905action}. In this experiment, the number of hours of sleep gained under two drugs (Dextro and Laevo hyoscyamine hydrobromide) was recorded for each patient. The quantity of interest is the paired difference in hours of sleep,
\[
d_i = (\text{Laevo})_i - (\text{Dextro})_i, \qquad i=1,\dots,10.
\]
These differences measure the additional sleep induced by the Laevo drug relative to the Dextro drug.

We test the point null hypothesis
\[
H_0:\mu_d=0 \qquad \text{vs} \qquad H_1:\mu_d\neq 0,
\]
where $\mu_d$ denotes the mean difference in hours of sleep.\\
The corresponding \(t\)-statistic is
$t\approx 4.062$ with two-sided \(p\)-value
$p \approx 0.00283$. The Least favorable Bayes factor in favor of \(H_0\) is
\[
B_{01}^{\mathrm{LF}}
=
\sqrt{\frac{n}{2}}
\left(\frac{S_1^2}{S_0^2}\right)^{n/2}
=
\sqrt{\frac{10}{2}}
\left(\frac{13.616}{38.58}\right)^5
\approx 0.01224.
\]
Using the intrinsic approximation 
we obtain
\[
B_{01}^{\mathrm{IP}}
\approx
\sqrt{20}\,(1+1.8334)^{-5}\,
\frac{1.8334}{1-e^{-1.8334}}
\approx 0.05344.
\]
Hence both Bayes factors provide strong evidence against \(H_0\). In fact LF Bayes factor can be used to reject $H_0$ since it its constructed from LF priors that are proper priors.

\subsection{One Way Anova (continuation)}
 As established, the extrema for the Bayes factor bounds are attained when the F-statistic of the training sample is zero i.e $F_{y^s(\ell)} = 0$  which implies that the between- roup sum of squares is exactly zero.
To derive the explicit forms of $\pi_0^{\mathrm{LF}}$ and $\pi_1^{\mathrm{LF}}$, we start from the  definition of the LF prior 
where $\pi_k^N(\theta_k, \sigma) \propto \sigma^{-(1+q_k)}$ is the generalized default prior. 

The likelihood of the training sample $y^s(\ell)$ under model $M_k$ is a multivariate normal distribution of dimension $n_{01} = m+1$. Using the standard orthogonal decomposition of the sum of squares, we can rewrite the exponent of the likelihood in terms of the minimal residual sum of squares and the least squares estimate $\hat{\theta}_k^*$
\begin{equation*} \label{likelihood_decomp}
f_k(y^s(\ell)\mid \theta_k, \sigma) \propto \frac{1}{\sigma^{m+1}} \exp\left\{ -\frac{1}{2\sigma^2} \Big[ \textbf{R}_k(y^s(\ell)) + (\theta_k - \hat{\theta}_k^*)^T \textbf{A}_k(\ell)^T \textbf{A}_k(\ell) (\theta_k - \hat{\theta}_k^*) \Big] \right\}
\end{equation*}

Recall that for the extreme training sample $y^s(\ell)$, the sample means of all $m$ groups are identical to a grand mean $\bar{y}^*$. Consequently, the residual sum of squares is minimized and identical for both models: $\textbf{R}_0(y^s(\ell)) = \textbf{R}_1(y^s(\ell)) = R^s$. \\
Under $M_0$,  the least squares estimate for the extreme training sample is the grand mean, $\hat{\mu}^* = \bar{y}^*$. The matrix $\textbf{A}_0(\ell)^T \textbf{A}_0(\ell)$ is a scalar equal to the size of the training sample, $n_{01} = m+1$.

Substituting and multiplying by the default prior $\pi_0^N(\mu, \sigma) \propto \sigma^{-(1+q_0)}$, we obtain
\[
\pi_0^{\mathrm{LF}}(\mu, \sigma) \propto \frac{1}{\sigma^{m + q_0 + 2}} \exp\left\{ -\frac{1}{2\sigma^2} \left[ R^* + (m+1)(\mu - \bar{y}^*)^2 \right] \right\}
\]
This is the kernel of a Normal-Inverse-Gamma distribution, it can be shown that gives the conditional prior $\mu \mid \sigma \sim N(\bar{y}^*, \sigma^2 / (m+1))$.
\\
Under $M_1$, the parameter vector is $\theta_1 = \boldsymbol{\mu} = (\mu_1, \dots, \mu_m)^T$. Because all group sample means in $y^s(\ell)$ are exactly $\bar{y}^*$, the least squares estimate vector is $\hat{\boldsymbol{\mu}}^* = \bar{y}^*\textbf{1}_m$, where $\textbf{1}_m$ is a vector of ones of length $m$.

The matrix $\textbf{A}_1(\ell)^T \textbf{A}_1(\ell)$ is a diagonal matrix containing the sample sizes of each group in the minimal training sample. Let $n_i(\ell)$ denote these sizes, where $n_i(\ell) = 1$ for the $m-1$ groups with a single observation, and $n_j(\ell) = 2$ for the single group with two observations. The quadratic form therefore simplifies to a sum of independent squared terms
$(\boldsymbol{\mu} - \hat{\boldsymbol{\mu}}^*)^T \textbf{A}_1(\ell)^T \textbf{A}_1(\ell) (\boldsymbol{\mu} - \hat{\boldsymbol{\mu}}^*) = \sum_{i=1}^m n_i(\ell)(\mu_i - \bar{y}^*)^2$

Again substituting this  and multiplying by the default prior we obtain
\[
\pi_1^{\mathrm{LF}}(\boldsymbol{\mu}, \sigma) \propto \frac{1}{\sigma^{m + q_1 + 2}} \exp\left\{ -\frac{1}{2\sigma^2} \left[ R^* + \sum_{i=1}^m n_i(\ell)(\mu_i - \bar{y}^*)^2 \right] \right\}
\]
This factorization implies that, conditionally on $\sigma$, the prior components $\mu_i$ are independent, yielding the multivariate normal prior where $\mu_i \mid \sigma \sim N(\bar{y}^*, \sigma^2/n_i(\ell))$ for all $i = 1, \dots, m$.

\section{Least Favorable Bayes Factors as Expanded Bounds}\label{sec12}

The key insight is the following relation:
\begin{equation}
B_{10}\bigl(y(-\ell)\mid y(\ell)\bigr)
\le
B_{10}^N(y)\max_{y(\ell)}B_{01}^N\!\bigl(y(\ell)\bigr)
\le
B_{10}^N(y)\sup_{y^*(\ell)}B_{01}^N\!\bigl(y^*(\ell)\bigr),
\end{equation}
 The first supremum is taken over training
samples extracted from the observed data, while the second is taken over the
full theoretical training-sample space.

If the extrema training sample is theoretical, we may formally expand the data
as if this sample had been observed. In that case, the posterior-prior
construction gives
\begin{equation}
B_{10}\bigl(y\mid y^*(\ell)\bigr)
=
B_{10}^N\bigl(y,y^*(\ell)\bigr)
\frac{m_0^N\!\bigl(y^*(\ell)\bigr)}
     {m_1^N\!\bigl(y^*(\ell)\bigr)}
=
B_{10}^N\bigl(y,y^*(\ell)\bigr)
B_{01}^N\!\bigl(y^*(\ell)\bigr).
\end{equation}
This expanded quantity is not necessarily equal to the original intrinsic
bound, since the latter uses \(B_{10}^N(y)\), whereas the expanded construction
uses \(B_{10}^N(y,y^*(\ell))\).

Now with the definition of the LF prior from Section~\ref{sec3}, if
\(y^*(\ell)\) is an extremal theoretical training sample we get the following proposition
\begin{proposition}
    If we assume conditional independence then 
    \[
    B_{10}^{LF}(y)=B_{10}(y|y^*(\ell))
    \]
\end{proposition}
\begin{proof}
\begin{equation*}
     B_{10}^{LF}(y) =\dfrac{\int f(y|\theta_1)\pi(\theta_1|y^*(\ell)) d\theta_1}{\int f(y|\theta_0)\pi(\theta_0|y^*(\ell)) d\theta_0}=\dfrac{\dfrac{\int f(y,y^*(\ell)|\theta_1)\pi^N(\theta_1)d\theta_1}{m_1^N(y^*(\ell))}}{\dfrac{\int f(y,y^*(\ell)|\theta_0)\pi^N(\theta_0)d\theta_0}{m_0^N(y^*(\ell))}}=\dfrac{m_1^N(y,y^*(\ell))}{m_0^N(y,y^*(\ell))} \cdot \dfrac{m_0^N(y^*(\ell))}{m_1^N(y^*(\ell))}
     \end{equation*}
\end{proof}
This important result states that the Least Favorable Bayes factor is equal to the 
expanded                    bound analogue of the intrinsic bound.
\subsection{Example:Normal Mean hypothesis test with $\sigma$ known (Continuation) }
If we compute the expaned Bayes Factor $B^N_{10}(y,\mu_0)$ we obtained 
\[
B_{10}^N(y,\mu_0)
=
\frac{\sqrt{2\pi}\sigma_0}{\sqrt{n+1}}
\exp\left\{
\frac{n^2}{2\sigma_0^2(n+1)}(\bar{y}-\mu_0)^2
\right\}.
\]
and if we multiply by $B_{01}^N(\mu_0)$ 
\[
\frac{\sqrt{2\pi}\sigma_0}{\sqrt{n+1}}
\exp\left\{
\frac{n^2}{2\sigma_0^2(n+1)}(\bar{y}-\mu_0)^2
\right\} \cdot \frac{1}{\sqrt{2\pi}\sigma_0}=\frac{1}{\sqrt{n+1}}
\exp\!\left\{\frac{n^2}{2\sigma_0^2(n+1)}\,(\bar{y}-\mu_0)^2\right\}
\]
which is the Least Favorable Bayes factor obtained beforehand.

\section{Conclusions}\label{sec13}
\begin{enumerate} 
\item Bounds in Testing Hypotheses are crucial for robust rejection of Null Hypotheses. If the bounds of the probabilities of the null are high, there is not enough evidence for rejection. Here we find bounds for all priors based on training samples.
\item The Bayesian Principle: Berger and Pericchi (1996) \cite{berger1996intrinsic}state that a statistical procedure is reasonable if i) there is a prior that generates it and ii) that prior is sensible. In this paper, we introduce the novel idea of a "Least Favorable Intrinsic Prior", which is proper and not too concentrated as unreasonable point masses. We also introduce the idea of an expanded sample, for which the imaginary training sample is added to the sample, and thus the bound becomes exact for the Least Favorable (Intrinsic) Prior.
 
\end{enumerate}

\backmatter

\bibliography{sn-bibliography}

\end{document}